\documentclass{article}

\usepackage{graphicx}
\usepackage{cite}
\usepackage{authblk}
\usepackage{amsthm}
\usepackage{amsmath}
\usepackage{amssymb}
\usepackage[left=3cm, right=3cm, lines=45, top=1.0in, bottom=1.0in]{geometry}
\usepackage[colorlinks=true]{hyperref}

\date{}
\title{Brunn--Minkowski Inequality for the First Complex $\sigma_{2}$-Hessian Eigenvalue}
\author{Chuanqiang Chen, Jiahuan Li, Xi-Nan Ma}

\newcommand{\keywords}[1]{\par\quad\textbf{Keywords:} #1}

\begin{document}

\maketitle
\newtheorem{theorem}{Theorem}[section]
\newtheorem{definition}[theorem]{Definition}
\newtheorem{lemma}[theorem]{Lemma}
\newtheorem{corollary}[theorem]{Corollary}
\newtheorem{example}[theorem]{Example}
\newtheorem{proposition}[theorem]{Proposition}
\newtheorem{conjecture}[theorem]{Conjecture}
\newtheorem{remark}[theorem]{Remark}
\newtheorem{assumption}[theorem]{Assumption}

\begin{abstract}
There are relatively few results on the convexity of solutions to complex equations. In this paper, We prove a strict real log-concavity theorem for the first eigenfunction of the complex $\sigma_{2}$-Hessian operator on smooth, bounded, real uniformly strictly convex domains in $\mathbb{C}^{n}$. As an application, we obtain a Brunn--Minkowski inequality for the first complex $\sigma_{2}$-Hessian eigenvalue. The proof combines a Bian--Guan constant-rank argument, a new inverse-convexity lemma for the compressed real Hessian, and Salani's viscosity admissible-test-function method.
\end{abstract}

\keywords{Complex Hessian equations, Brunn--Minkowski inequality, log-concavity, constant rank theorem, inverse convexity}

\numberwithin{equation}{section}

\section{Introduction}

Convexity properties of solutions of elliptic equations play a central role in the interaction between partial differential equations and convex geometry. A classical starting point is the result of Brascamp and Lieb \cite{brascamp1976extensions}, which implies the log-concavity of the first Dirichlet eigenfunction of the Laplacian on convex domains and leads to a Brunn--Minkowski inequality for the first Laplace eigenvalue. Since then, many analogous questions have been studied for nonlinear elliptic equations, especially for Hessian equations and Monge--Amp$\grave{e}$re type equations; see for instance \cite{colesanti2005brunn,salani2005monge,salani2012convexity,limasa} and the references therein.

For real Hessian equations, one considers
\begin{equation}\label{eq:real-hessian}
\sigma_{k}(D^{2}u)=\lambda(\Omega)(-u)^{k},\qquad
u<0\ \text{in }\Omega,\qquad
u=0\ \text{on }\partial\Omega,
\end{equation}
where $\sigma_{k}$ denotes the $k$-th elementary symmetric function of the eigenvalues of the real Hessian. The admissible branch is the G{\aa}rding cone $\Gamma_{k}$. Wang \cite{wang1994class} established existence and uniqueness, up to positive multiples, of the first eigenfunction for the real $k$-Hessian problem in smooth convex domains. The Brunn--Minkowski theory for the corresponding eigenvalue is known for $k=1$ and $k=n$, and for intermediate $k$ only in special cases. In particular, the real $\sigma_{2}$ case is delicate. Liu--Ma--Xu and Salani treated the three-dimensional case by using convexity of suitable power or logarithmic transforms and the Minkowski addition of convex functions \cite{liu2010brunn,salani2012convexity}. A higher-dimensional proof for the real $\sigma_{2}$ eigenvalue relies on a constant-rank theorem and a refined matrix inverse-convexity estimate.

The present paper develops the analogous theory for the complex $\sigma_{2}$-Hessian eigenvalue. Let $\Omega\subset\mathbb{C}^{n}$ be a bounded smooth domain. For a real-valued function $u$, write
\[
(u_{i\bar{j}})_{1\leq i,j\leq n}
\]
for its complex Hessian. We study the first eigenvalue problem
\begin{equation}\label{eq:complex-eigen}
\sigma_{2}(u_{i\bar{j}})=\lambda_{\mathbb{C},2}(\Omega)(-u)^{2},\qquad
u<0\ \text{in }\Omega,\qquad
u=0\ \text{on }\partial\Omega,
\end{equation}
with $A_{u}\in\Gamma_{2}$. The corresponding Rayleigh quotient is
\begin{equation}\label{eq:rayleigh}
\lambda_{\mathbb{C},2}(\Omega)
=
\inf_{\varphi\in\mathcal{A}_{2}(\Omega)}
\frac{\displaystyle\int_{\Omega}(-\varphi)\sigma_{2}(\varphi_{i\bar{j}})\,dV}
{\displaystyle\int_{\Omega}(-\varphi)^{3}\,dV},
\end{equation}
where $\mathcal{A}_{2}(\Omega)$ denotes the class of negative $2$-admissible functions vanishing on the boundary. For the spectral theory of complex Hessian operators we use the variational existence theory of Badiane--Zeriahi \cite{badiane2023variational}, the $C^{1,1}$-regularity and uniqueness theorem of Chu--Liu--McCleerey \cite{chu2024eigenvalue}, and standard interior regularity for strictly admissible complex Hessian equations; see also \cite{dinew2011apriori,dinew2012liouville,hou2010second,zdk}.

Our first main result is the strict real log-concavity of the first eigenfunction.

\begin{theorem}\label{thm:log-concavity}
Let $\Omega\subset\mathbb{C}^{n}$, $n\geq2$, be a bounded, smooth, real uniformly strictly convex domain. Let $u$ be the first $2$-admissible eigenfunction of \eqref{eq:complex-eigen}. Then
\[
v=-\log(-u)
\]
is strictly convex as a function on the real domain $\Omega\subset\mathbb{R}^{2n}$.
\end{theorem}

Here real uniformly strictly convex means that the second fundamental form of $\partial\Omega$, considered as a hypersurface in $\mathbb{R}^{2n}$, is positive definite with a uniform positive lower bound. Since $\partial\Omega$ is compact, this is automatic if the boundary is $C^{2}$ and the real second fundamental form is positive at every boundary point.

The proof of Theorem \ref{thm:log-concavity} follows the classical Caffarelli--Friedman--Korevaar--Lewis deformation strategy. The key new point is the complex inverse-convexity structure needed for the Bian--Guan constant-rank theorem. Under the logarithmic transform $u=-e^{-v}$, equation \eqref{eq:complex-eigen} becomes
\begin{equation}\label{eq:transformed-intro}
\sigma_{2}(v_{i\bar{j}}-v_{i}v_{\bar{j}})=\lambda_{\mathbb{C},2}(\Omega).
\end{equation}
Thus the relevant operator contains the rank-one term $\partial v\otimes\overline{\partial v}$. We prove that, for every fixed complex vector $q$ and every $\lambda>0$, the set
\begin{equation}\label{eq:intro-level}
\left\{
A\in S^{2n}_{++}:
\sigma_{2}\bigl(C(A^{-1})-qq^{*}\bigr)\leq\lambda
\right\}
\end{equation}
is convex. Here $C$ is the linear map which extracts the complex Hessian part from a real Hessian. This level-set convexity verifies the Bian--Guan structural condition \cite{bian2009microscopic,bian2010structural}. Once convexity of $v$ is known, the constant-rank theorem implies that $D^{2}_{\mathbb{R}}v$ has constant rank. Boundary strict convexity and a deformation from the ball then yield strict convexity everywhere.

The second main result is a Brunn--Minkowski inequality for the eigenvalue.

\begin{theorem}[Brunn--Minkowski inequality]\label{thm:BM}
Let $\Omega_{0},\Omega_{1}\subset\mathbb{C}^{n}$, $n\geq2$, be bounded, smooth, real uniformly strictly convex domains and let $t\in[0,1]$. Set
\[
\Omega_{t}=(1-t)\Omega_{0}+t\Omega_{1}.
\]
Then
\begin{equation}\label{eq:BM}
\lambda_{\mathbb{C},2}(\Omega_{t})^{-1/4}
\geq
(1-t)\lambda_{\mathbb{C},2}(\Omega_{0})^{-1/4}
+t\lambda_{\mathbb{C},2}(\Omega_{1})^{-1/4}.
\end{equation}
\end{theorem}

The exponent $-1/4$ is dictated by the scaling law
\[
\lambda_{\mathbb{C},2}(r\Omega)=r^{-4}\lambda_{\mathbb{C},2}(\Omega),\qquad r>0,
\]
because the complex Hessian has second order and $\sigma_{2}$ is homogeneous of degree two.

The proof of Theorem \ref{thm:BM} is modeled on Colesanti--Salani's infimal-convolution method \cite{colesanti2005brunn,salani2012convexity}. Let $v_{i}=-\log(-u_{i})$ be the transforms of the first eigenfunctions on $\Omega_{i}$. Since the $v_{i}$ are strictly convex by Theorem \ref{thm:log-concavity}, one can define their infimal convolution
\[
w(z)=\inf\{(1-t)v_{0}(x)+tv_{1}(y):z=(1-t)x+ty\}.
\]
At corresponding points the gradients agree and the real Hessian of $w$ is the harmonic mean of the real Hessians of $v_{0}$ and $v_{1}$. The inverse-convexity \eqref{eq:intro-level} then gives the correct pointwise inequality for $w$.

A subtle issue is that $\tilde{u}=-e^{-w}$ need not be classically admissible. Salani's viscosity argument avoids this difficulty: one proves instead that $\tilde{u}$ is a viscosity supersolution when only admissible test functions are used. This is precisely the mechanism used in \cite{salani2012convexity}; see also the viscosity treatment of Hessian equations in \cite{colesanti1999hessian,urbas1990nonclassical}. We include a self-contained admissible-replacement argument below, which avoids inserting a non-admissible function directly into the Rayleigh quotient.
\section*{Acknowledgments}
Chuanqiang Chen was supported by ZJNSF No. LRG25A010002 and NSFC No. 12571219. Xi-Nan Ma and Jiahuan Li were supported by the National Natural Science Foundation of China [grant number 2025YFA1017601]. 

\textbf{Notations.} In this article, we use the following notations.
\begin{itemize}
    \item For a complex-valued function $u$, we write
    $u_{i}:=\frac{\partial u}{\partial z_{i}}$,
    $u_{\bar{j}}:=\frac{\partial u}{\partial \bar{z}_{j}}$, and
    $u_{i\bar{j}}:=\frac{\partial^{2}u}{\partial \bar{z}_{j}\partial z_{i}}$.
    \item The transpose of a matrix $A$ is denoted by $A^T$.
    \item The conjugate transpose of a matrix $A$ is denoted by $A^*$.
    \item $S_{++}^{k}:=\left\{ A\in R^{k\times k},A=A^T,A>0 \right\}$.
    \item $H_{++}^{k}:=\left\{ A\in C^{k\times k},A=A^{*},A>0 \right\}$
\end{itemize}

\section{Preliminaries}

\subsection{Complex Hessians as compressions of real Hessians}

Identify $\mathbb{C}^{n}$ with $\mathbb{R}^{2n}$ by writing
\[
z_{j}=x_{j}+iy_{j},\qquad j=1,\ldots,n.
\]
For a real symmetric $2n\times2n$ matrix
\[
H=
\begin{pmatrix}
U & V\\
V^{T} & W
\end{pmatrix},
\]
where the blocks correspond to the $x$ and $y$ variables, define
\begin{equation}\label{eq:C-def}
C(H)=\frac14\bigl(U+W+i(V-V^{T})\bigr).
\end{equation}
Equivalently, let $E$ be the $n\times n$ identity matrix. If
\[
R=
\begin{pmatrix}
E\\
iE
\end{pmatrix}
\in\mathbb{C}^{2n\times n},
\]
then
\begin{equation}\label{eq:C-R}
C(H)=\frac14R^{*}HR.
\end{equation}
If $H=D^{2}_{\mathbb{R}}u$, then $C(H)=(u_{i\bar{j}})$.

For a real vector $\xi=(a,b)\in\mathbb{R}^{n}\times\mathbb{R}^{n}$, let $\eta=a-ib\in\mathbb{C}^{n}$. Then
\begin{equation}\label{eq:rank-one-compression}
C(\xi\otimes\xi)=\frac14\eta\eta^{*}.
\end{equation}
In particular, a nonzero real rank-one positive semidefinite direction is mapped to a nonzero Hermitian rank-one positive semidefinite direction.

\subsection{Admissibility}

For a Hermitian matrix $M$, denote by $\lambda(M)$ its eigenvalue vector. We write
\[
M\in\Gamma_{2}
\]
if
\[
\sigma_{1}(M)>0,\qquad \sigma_{2}(M)>0.
\]
A real function $u$ is called complex $2$-admissible if
\[
(u_{i\bar{j}})(z)\in\Gamma_{2}
\]
for every $z$ in the domain.

The first Newton tensor of a Hermitian matrix $M$ is
\[
T_{1}(M)=\sigma_{1}(M)I-M.
\]
On $\Gamma_{2}$ one has
\begin{equation}\label{eq:T1-positive}
T_{1}(M)>0.
\end{equation}
Indeed, in a basis diagonalizing $M$, the eigenvalues of $T_{1}(M)$ are $\sum_{j\neq i}\lambda_{j}(M)$, which are positive on $\Gamma_{2}$.

\subsection{Spectral input for the complex Hessian eigenvalue}

We use the following standard analytic input. It is not proved here. The existence of a first eigenfunction and a variational eigenvalue for complex Hessian operators can be obtained from Badiane--Zeriahi \cite{badiane2023variational}; $C^{1,1}$-regularity and uniqueness up to constants are proved in the strongly $m$-pseudoconvex setting by Chu--Liu--McCleerey \cite{chu2024eigenvalue}. Interior regularity for strictly admissible complex Hessian equations follows from the usual Evans--Krylov--Schauder theory once uniform ellipticity is available; see also \cite{dinew2011apriori,dinew2012liouville,hou2010second}.

\begin{theorem}[Spectral and regularity package]\label{ass:spectral}
For every bounded, smooth, real uniformly strictly convex domain $\Omega\subset\mathbb{C}^{n}$, the first complex $\sigma_{2}$-Hessian eigenvalue problem \eqref{eq:complex-eigen} has a negative $2$-admissible eigenfunction $u$, unique up to positive multiplicative constants. Moreover
\[
u\in C^{1,1}\left( \bar{\Omega} \right)\cap C^{\infty}(\Omega),
\]
and the normalized eigenfunctions depend continuously in $C^{2}_{\mathrm{loc}}$ along smooth uniformly controlled deformations of the domain.
\end{theorem}

\begin{remark}
For completeness, we justify this theorem in the appendix.
\end{remark}

\subsection{Bian--Guan level-set constant-rank theorem}

We use the following applied form of the Bian--Guan level-set constant-rank theorem \cite{bian2009microscopic,bian2010structural}.

\begin{theorem}\label{thm:BG}
Let $D\subset\mathbb{R}^{N}$ be connected. Let $v\in C^{3,1}(D)$ be a convex solution of
\[
F(D^{2}v,Dv,v,x)=0,
\]
where $F\in C^{2,1}$ in its variables. Assume:
\begin{enumerate}
\item $F$ is elliptic along the solution, i.e. $(F^{\alpha\beta})>0$;
\item $F(0,Dv,v,x)\neq0$ along the solution;
\item for every fixed relevant gradient $p$, the set
\[
\{(A,z,x)\in S^{N}_{++}\times\mathbb{R}\times D:F(A^{-1},p,z,x)\leq0\}
\]
is locally convex near the relevant points.
\end{enumerate}
Then $\operatorname{rank}D^{2}v$ is constant in $D$.
\end{theorem}

\section{Inverse convexity for compressed Hessians}

This section contains the algebraic core of the paper. And the calculations in this section can be referenced from the \cite{limasa}.

\subsection{A compressed inverse log-det lemma}

\begin{lemma}\label{lem:compressed-logdet}
Let $H\in S^{N}_{++}$ and let $X\in\mathbb{C}^{N\times m}$ have full column rank. Then
\[
H\longmapsto \log\det(X^{*}H^{-1}X)
\]
is convex on $S^{N}_{++}$.
\end{lemma}

\begin{proof}
Fix a real symmetric direction $K$ and set $H(t)=H+tK$. Let $S=H^{-1}$ and
\[
M(t)=X^{*}H(t)^{-1}X,\qquad \phi(t)=\log\det M(t).
\]
It suffices to prove $\phi''(0)\geq0$. The inverse matrix identities are
\[
\frac{d}{dt}H(t)^{-1}=-H(t)^{-1}KH(t)^{-1},
\]
and
\[
\frac{d^{2}}{dt^{2}}H(t)^{-1}=2H(t)^{-1}KH(t)^{-1}KH(t)^{-1}.
\]
Thus, at $t=0$,
\[
M(0)=A:=X^{*}SX,\qquad
M'(0)=-B,\quad B:=X^{*}SKSX,
\]
and
\[
M''(0)=2C,\qquad C:=X^{*}SKSKSX.
\]
The standard formula for $\log\det$ gives
\[
\phi''(0)=2\operatorname{tr}(A^{-1}C)-\operatorname{tr}(A^{-1}BA^{-1}B).
\]
Let $Y=H^{-1/2}X$ and $L=H^{-1/2}KH^{-1/2}$. Then
\[
A=Y^{*}Y,\qquad B=Y^{*}LY,\qquad C=Y^{*}L^{2}Y.
\]
The Gram matrix
\[
\begin{pmatrix}
Y^{*}Y & Y^{*}LY\\
Y^{*}LY & Y^{*}L^{2}Y
\end{pmatrix}
=
\left( Y,LY \right)^{*} \left( Y,LY \right)
\]
is positive semidefinite. Its Schur complement yields
\[
C\geq BA^{-1}B.
\]
Therefore
\[
\phi''(0)\geq
2\operatorname{tr}(A^{-1}BA^{-1}B)-\operatorname{tr}(A^{-1}BA^{-1}B)
=\operatorname{tr}(A^{-1}BA^{-1}B)\geq0.
\]
This proves convexity.
\end{proof}

\begin{proposition}[Inverse log-convexity]\label{prop:inverse-log}
For every $1\leq k\leq n$, the function
\[
H\longmapsto \log\sigma_{k}(C(H^{-1}))
\]
is convex on $S^{2n}_{++}$.
\end{proposition}

\begin{proof}
Let $X=R/2$, so that $C(H^{-1})=X^{*}H^{-1}X$. For each $k$-element subset $I\subset\{1,\ldots,n\}$, let $E_{I}$ be the coordinate inclusion and set $X_{I}=XE_{I}$. Then
\[
\det(C(H^{-1})_{I,I})=\det(X_{I}^{*}H^{-1}X_{I}).
\]
Since
\[
\sigma_{k}(C(H^{-1}))
=
\sum_{|I|=k}\det(X_{I}^{*}H^{-1}X_{I}),
\]
we have
\[
\log\sigma_{k}(C(H^{-1}))
=
\log\sum_{|I|=k}
\exp\left(\log\det(X_{I}^{*}H^{-1}X_{I})\right).
\]
Each term inside the log-sum-exp is convex by Lemma \ref{lem:compressed-logdet}; the log-sum-exp of convex functions is convex. This proves the claim.
\end{proof}

\subsection{The rank-one level-set convexity}

\begin{theorem}[Rank-one inverse level-set convexity]\label{thm:rank-one}
For every $q\in\mathbb{C}^{n}$ and every $\lambda>0$, the set
\begin{equation}\label{eq:Kqlambda}
K_{q,\lambda}
=
\left\{
A\in S^{2n}_{++}:
\sigma_{2}\bigl(C(A^{-1})-qq^{*}\bigr)\leq\lambda
\right\}
\end{equation}
is convex.
\end{theorem}

\begin{proof}
First assume $q\neq0$. By unitary invariance, we may suppose
\[
q=\rho e,\qquad \rho=|q|>0.
\]
Set
\[
P=I-ee^{*}.
\]
For $C=C(A^{-1})$, the rank-one perturbation formula gives
\begin{equation}\label{eq:rank-one-perturb}
\sigma_{2}(C-\rho^{2}ee^{*})
=
\sigma_{2}(C)-\rho^{2}e^{*}T_{1}(C)e
=
\sigma_{2}(C)-\rho^{2}\operatorname{tr}(PC).
\end{equation}
Since $C>0$ and $P\neq0$, one has $\operatorname{tr}(PC)>0$. Hence the condition in \eqref{eq:Kqlambda} is equivalent to
\begin{equation}\label{eq:J-level}
J_{\lambda}(A):=
\frac{\sigma_{2}(C(A^{-1}))-\lambda}
{\operatorname{tr}(PC(A^{-1}))}
\leq\rho^{2}.
\end{equation}
It remains to prove that $J_{\lambda}$ is convex.

Let $X=R/2$, so $C=X^{*}A^{-1}X$. Put
\[
\Theta=ee^{*}+\frac12P,\qquad
\mathcal{P}=XPX^{*},\qquad
\mathcal{R}=X\Theta X^{*}.
\]
Let
\[
g(A)=D_{\mathcal{P}}\det(A).
\]
Then
\[
g(A)=\det A\,\operatorname{tr}(PC).
\]
Using the second derivative formula for the determinant,
\[
D_{\mathcal{R}}g(A)
=
D_{\mathcal{R}}D_{\mathcal{P}}\det(A)
=
\det A\,[\operatorname{tr}(PC)\operatorname{tr}(\Theta C)-\operatorname{tr}(PC\Theta C)].
\]
Let $s=e^{*}Ce$. A direct computation using $P=I-ee^{*}$ and $\Theta=ee^{*}+\frac12P$ gives
\[
\operatorname{tr}(PC)\operatorname{tr}(\Theta C)
=\frac{1}{2}\bigl((\operatorname{tr}C)^{2}-s^{2}\bigr), \,\,\,\,
\operatorname{tr}(PC\Theta C)
=\frac{1}{2}\bigl(\operatorname{tr}(C^{2})-s^{2}\bigr).
\]
Hence
\[
\operatorname{tr}(PC)\operatorname{tr}(\Theta C)-\operatorname{tr}(PC\Theta C)
=\frac{1}{2}\bigl((\operatorname{tr}C)^{2}-\operatorname{tr}(C^{2})\bigr)
=\sigma_2(C).
\]
Therefore
\[
D_{\mathcal{R}}g(A)=\det A\,\sigma_{2}(C).
\]
Consequently
\begin{equation}\label{eq:Phi-compressed}
\Phi(A):=
\frac{\operatorname{tr}(PC)}{\sigma_{2}(C)}
=
\frac{g(A)}{D_{\mathcal{R}}g(A)}.
\end{equation}

Now we extend the definition of $\phi$ from $S^{2n}_{++}$ to $H^{2n}_{++}$. The determinant is hyperbolic on the positive definite cone. Directional derivatives in directions in the closure of the cone preserve hyperbolicity, and their cones contain $H^{2n}_{++}$; see G{\aa}rding \cite{garding1959inequality}, Bauschke--Guler--Lewis--Sendov \cite{bauschke2001hyperbolic}, and Renegar \cite{renegar2006hyperbolic}. Applying G{\aa}rding's quotient concavity theorem to $g$ in the direction $\mathcal{R}$, with an approximation by positive definite directions if necessary, we obtain that $\Phi$ is concave on $H^{2n}_{++}$. Since $S^{2n}_{++}$ is included in $H^{2n}_{++}$ and $S^{2n}_{++}$ is convex, we know that $\Phi$ is concave on $S^{2n}_{++}$.

Since $\Phi>0$, the reciprocal
\[
\frac{\sigma_{2}(C)}{\operatorname{tr}(PC)}
\]
is convex. Moreover,
\[
L(A):=\operatorname{tr}(PC(A^{-1}))=\operatorname{tr}(\mathcal{P}A^{-1}).
\]
The Alvarez--Lasry--Lions lemma \cite{alvarez1997convex} gives that $A\mapsto1/L(A)$ is concave. Hence
\[
-\lambda\frac{1}{L(A)}
\]
is convex. Thus
\[
J_{\lambda}(A)=
\frac{\sigma_{2}(C)}{\operatorname{tr}(PC)}
-\frac{\lambda}{\operatorname{tr}(PC)}
\]
is convex. The sublevel set \eqref{eq:J-level} is therefore convex.

Now let $q=0$. Fix any unit vector $e_{n}$ and set $q_{\varepsilon}=\varepsilon e_{n}$. By the previous step, $K_{q_{\varepsilon},\lambda}$ is convex for every $\varepsilon>0$. We claim
\[
K_{0,\lambda}=\bigcap_{\varepsilon>0}K_{q_{\varepsilon},\lambda}.
\]
Indeed, the condition $A\in K_{q_{\varepsilon},\lambda}$ is
\[
\sigma_{2}(C(A^{-1}))-\varepsilon^{2}\operatorname{tr}(PC(A^{-1}))
\leq\lambda.
\]
Letting $\varepsilon\downarrow0$ gives the claim. An arbitrary intersection of convex sets is convex. Thus $K_{0,\lambda}$ is convex.
\end{proof}

\section{A constant-rank theorem for the transformed equation}

Consider the transformed equation
\begin{equation}\label{eq:transformed}
\sigma_{2}(v_{i\bar{j}}-v_{i}v_{\bar{j}})=\lambda.
\end{equation}
In real notation this is
\[
F(D^{2}_{\mathbb{R}}v,Dv)=0,
\]
where
\begin{equation}\label{eq:F-complex}
F(H,p)=\sigma_{2}\bigl(C(H)-q(p)q(p)^{*}\bigr)-\lambda,
\end{equation}
and $q(p)$ denotes the complex gradient associated to the real gradient $p=Dv$.

\begin{theorem}[Complex $\sigma_{2}$ constant rank]\label{thm:complex-rank}
Let $D\subset\mathbb{C}^{n}$ be connected. Let $v\in C^{4}(D)$ solve \eqref{eq:transformed}. Assume
\[
D^{2}_{\mathbb{R}}v\geq0
\]
and
\[
N:=C(D^{2}_{\mathbb{R}}v)-\partial v\otimes\bar{\partial}v\in\Gamma_{2}.
\]
Then $\operatorname{rank}D^{2}_{\mathbb{R}}v$ is constant in $D$.
\end{theorem}

\begin{proof}
We verify the assumptions of Theorem \ref{thm:BG}.

First, ellipticity. Let $\xi=(a,b)\in\mathbb{R}^{2n}$ be nonzero and put $\eta=a-ib$. By \eqref{eq:rank-one-compression},
\[
C(\xi\otimes\xi)=\frac14\eta\eta^{*}.
\]
Differentiating \eqref{eq:F-complex} with respect to $H$ in the direction $\xi\otimes\xi$ gives
\[
D_{H}F(H,p)[\xi\otimes\xi]
=
\operatorname{tr}(T_{1}(N)C(\xi\otimes\xi))
=
\frac14\eta^{*}T_{1}(N)\eta.
\]
Since $N\in\Gamma_{2}$, $T_{1}(N)>0$. Hence $D_{H}F[\xi\otimes\xi]>0$ for every $\xi\neq0$, which is strict ellipticity in the real Hessian variables.

Second, nondegeneracy. Since $qq^{*}$ has rank one,
\[
\sigma_{2}(-qq^{*})=0.
\]
Thus
\[
F(0,p)=-\lambda\neq0.
\]

Third, the Bian--Guan level-set condition is precisely Theorem \ref{thm:rank-one}, with $A\in S^{2n}_{++}$ and fixed $q=q(p)$:
\[
\{A\in S^{2n}_{++}:F(A^{-1},p)\leq0\}=K_{q,\lambda}.
\]
It is convex, hence locally convex. The constant-rank conclusion follows from Theorem \ref{thm:BG}.
\end{proof}

\section{Proof of Theorem \ref{thm:log-concavity}}

\subsection{The ball case}

\begin{lemma}[The ball]\label{lem:ball}
Let $B_{R}\subset\mathbb{C}^{n}$ be a Euclidean ball and let $u<0$ be its first complex $\sigma_{2}$-eigenfunction. Then $v=-\log(-u)$ is strictly convex as a real function in $B_{R}$.
\end{lemma}

\begin{proof}
By uniqueness of the first eigenfunction and invariance of the ball and of $\sigma_{2}(u_{i\bar{j}})$ under unitary transformations, $u$ is radial:
\[
u(z)=\varphi(r),\qquad r=|z|.
\]
We have $\varphi<0$ in $[0,R)$, $\varphi(R)=0$, and $\varphi'(r)>0$ for $0<r<R$.

For a radial function, the complex Hessian has one radial eigenvalue
\[
a(r)=\frac14\left(\varphi''(r)+\frac{\varphi'(r)}{r}\right)
\]
and $n-1$ tangential eigenvalues
\[
b(r)=\frac{\varphi'(r)}{2r}.
\]
Thus
\begin{equation}\label{eq:radial-sigma}
\sigma_{2}(u_{i\bar{j}})
=(n-1)ab+\binom{n-1}{2}b^{2}
=
\frac{n-1}{8}\left(\frac{\varphi'\varphi''}{r}
+(n-1)\frac{(\varphi')^{2}}{r^{2}}\right).
\end{equation}
Let $v=-\log(-\varphi)$, so
\[
\varphi=-e^{-v},\qquad
\varphi'=e^{-v}v',\qquad
\varphi''=e^{-v}(v''-(v')^{2}).
\]
Since
\[
\sigma_{2}(u_{i\bar{j}})=\lambda(-u)^{2}=\lambda e^{-2v},
\]
\eqref{eq:radial-sigma} becomes
\begin{equation}\label{eq:radial-v}
rv'v''-r(v')^{3}+(n-1)(v')^{2}
=
\frac{8\lambda}{n-1}r^{2}.
\end{equation}
Set
\[
c=\frac{8\lambda}{n-1}>0.
\]
Then \eqref{eq:radial-v} reads
\[
rv'v''-r(v')^{3}+(n-1)(v')^{2}=cr^{2}.
\]
Since $\varphi'>0$, we have $v'>0$ on $(0,R)$.

At $r=0$, write
\[
L=\lim_{r\to0^{+}}\frac{v'(r)}{r}=v''(0).
\]
The leading terms in the last equation give $nL^{2}=c$, so
\[
v''(0)=\sqrt{\frac{8\lambda}{n(n-1)}}>0.
\]
Suppose that $v''$ has a first zero at $r_{0}\in(0,R)$. Then $v''>0$ on $[0,r_{0})$ and $v'''(r_{0})\leq0$. Differentiating \eqref{eq:radial-v} and using $v''(r_{0})=0$ yields
\[
r_{0}v'(r_{0})v'''(r_{0})-(v'(r_{0}))^{3}=2cr_{0}.
\]
Hence
\[
v'''(r_{0})=\frac{2c}{v'(r_{0})}+\frac{(v'(r_{0}))^{2}}{r_{0}}>0,
\]
a contradiction. Thus $v''>0$ on $[0,R)$.

The real Hessian of a radial function has eigenvalues $v''$ in the radial direction and $v'/r$ in tangential directions. Since $v''>0$ and $v'>0$ for $r>0$, $D^{2}_{\mathbb{R}}v>0$ in the ball.
\end{proof}

\subsection{Boundary convexity}
We record the following well-known boundary convexity lemma; see, for example, Caffarelli and Friedman \cite{caffarelli1985convexity} or Korevaar \cite{korevaar1983convex}.

\begin{lemma}\label{lem:boundary1}
Let $\Omega\subset\mathbb{R}^{n}$ be smooth, bounded, and strictly convex. Let $u\in C^{\infty}(\Omega)\cap C^{1,1}(\overline{\Omega})$ satisfy
\begin{equation}\label{eq:boundary-data}
u<0\ \text{in }\Omega,\qquad u=0\ \text{and}\ Du\cdot\nu>0\ \text{on }\partial\Omega,
\end{equation}
where $\nu$ is the exterior normal to $\partial\Omega$. Let
\[
\Omega_{\varepsilon}=\{x\in\Omega:d(x,\partial\Omega)>\varepsilon\}
\]
and let $v=f(u)$. Then, for sufficiently small $\varepsilon>0$, the function $v$ is strictly convex in the boundary strip $\Omega\setminus\Omega_{\varepsilon}$, provided that $f$ satisfies
\[
f'>0,\qquad f''>0,\qquad \lim_{u\to0^{-}}\frac{f'}{f''}=0.
\]
\end{lemma}

\begin{remark}
In Korevaar \cite{korevaar1983convex}, the function $u$ belongs to $C^{2}(\overline{\Omega})$. However, the calculation in Caffarelli and Friedman \cite{caffarelli1985convexity} also shows that the same conclusion holds in our setting, where $u\in C^{\infty}(\Omega)\cap C^{1,1}(\overline{\Omega})$.
\end{remark}
Applying Lemma \ref{lem:boundary1} gives the following consequence.
\begin{lemma}[Boundary strict convexity]\label{lem:boundary}
Let $\Omega\subset\mathbb{R}^{N}$ be bounded, smooth, and real uniformly strictly convex. Let $u<0$ in $\Omega$, $u=0$ on $\partial\Omega$, and $\partial_{\nu}u>0$ on $\partial\Omega$, where $\nu$ is the exterior unit normal. Then
\[
v=-\log(-u)
\]
is strictly convex in a boundary strip.
\end{lemma}

\begin{lemma}[Hopf nondegeneracy]\label{lem:hopf}
For the eigenfunction $u$ of \eqref{eq:complex-eigen}, one has
\[
\partial_{\nu}u>0\quad\text{on }\partial\Omega.
\]
\end{lemma}

\begin{proof}
Since $u$ is $2$-admissible,
\[
\sigma_{1}(u_{i\bar{j}})>0.
\]
Therefore
\[
\Delta_{\mathbb{R}}u=4\sum_{j=1}^{n}u_{j\bar{j}}
=4\sigma_{1}(u_{i\bar{j}})>0.
\]
Thus $u$ is strictly subharmonic. Since $u<0$ in $\Omega$ and $u=0$ on $\partial\Omega$, the Hopf boundary point lemma gives $\partial_{\nu}u>0$.
\end{proof}

\subsection{The deformation argument}

\begin{proof}[Proof of Theorem \ref{thm:log-concavity}]
Choose a ball $B_{R}$ and connect it to $\Omega$ by Minkowski addition:
\[
\Omega_{s}=(1-s)B_{R}+s\Omega,\qquad 0\leq s\leq1.
\]
In terms of support functions this is
\[
h_{s}=(1-s)R+sh_{\Omega}.
\]
Since $B_{R}$ and $\Omega$ are real uniformly strictly convex, the curvature-radius matrices
\[
\nabla^{2}_{S}h_{s}+h_{s}g_{S}
\]
remain uniformly positive. Thus the domains $\Omega_{s}$ stay in the same smooth, real uniformly strictly convex class.

Let $u_{s}$ be the normalized first eigenfunction in $\Omega_{s}$, and set
\[
v_{s}=-\log(-u_{s}).
\]
Define
\[
I=\{s\in[0,1]:D^{2}_{\mathbb{R}}v_{s}>0\ \text{in }\Omega_{s}\}.
\]
By Lemma \ref{lem:ball}, $0\in I$.

Openness follows from the $C^{2}_{\mathrm{loc}}$ continuous dependence in Theorem \ref{ass:spectral}, together with the boundary strip strict convexity from Lemmas \ref{lem:boundary} and \ref{lem:hopf}.

For closedness, let $s_{j}\in I$ and $s_{j}\to s_{0}$. By the local $C^{2}$ convergence of $v_{s_{j}}$ to $v_{s_{0}}$ on compact subsets,
\[
D^{2}_{\mathbb{R}}v_{s_{0}}\geq0.
\]
The function $v_{s_{0}}$ satisfies
\[
\sigma_{2}((v_{s_{0}})_{i\bar{j}}-(v_{s_{0}})_{i}(v_{s_{0}})_{\bar{j}})
=
\lambda_{\mathbb{C},2}(\Omega_{s_{0}}),
\]
and is admissible in the transformed sense. Hence Theorem \ref{thm:complex-rank} gives that $\operatorname{rank}D^{2}_{\mathbb{R}}v_{s_{0}}$ is constant in $\Omega_{s_{0}}$. By Lemmas \ref{lem:boundary} and \ref{lem:hopf}, $D^{2}_{\mathbb{R}}v_{s_{0}}>0$ in a boundary strip. Therefore the constant rank must be full, and
\[
D^{2}_{\mathbb{R}}v_{s_{0}}>0
\]
in all of $\Omega_{s_{0}}$. Thus $s_{0}\in I$.

Since $I$ is non-empty, open, and closed in $[0,1]$, we have $I=[0,1]$. In particular, $v_{1}=-\log(-u)$ is strictly convex in $\Omega$.
\end{proof}

\section{The Brunn--Minkowski inequality}

\subsection{Infimal convolution}

Let $\Omega_{0},\Omega_{1}$ be as in Theorem \ref{thm:BM}. Let $u_{i}$ be first eigenfunctions and
\[
v_{i}=-\log(-u_{i}),\qquad i=0,1.
\]
By Theorem \ref{thm:log-concavity}, $v_{i}$ are strictly convex and $v_{i}\to+\infty$ at the boundary. Define
\[
\Omega_{t}=(1-t)\Omega_{0}+t\Omega_{1}
\]
and
\begin{equation}\label{eq:inf-conv}
w(z)=\inf\{(1-t)v_{0}(x)+tv_{1}(y):z=(1-t)x+ty\}.
\end{equation}
Then $w$ is strictly convex and tends to $+\infty$ on $\partial\Omega_{t}$. Moreover, if $z=(1-t)x+ty$ realizes the infimum, then
\begin{equation}\label{eq:grad-agree}
Dv_{0}(x)=Dv_{1}(y)=Dw(z).
\end{equation}
Furthermore,
\begin{equation}\label{eq:hessian-harmonic}
D^{2}_{\mathbb{R}}w(z)^{-1}
=(1-t)D^{2}_{\mathbb{R}}v_{0}(x)^{-1}+
tD^{2}_{\mathbb{R}}v_{1}(y)^{-1}.
\end{equation}
These are standard facts about infimal convolution and Legendre transforms; see Rockafellar \cite{rockafellar1970convex} and Salani \cite{salani2012convexity}.

\subsection{Viscosity supersolution property}

\begin{lemma}\label{lem:viscosity}
Let
\[
\Lambda=\max\{\lambda_{\mathbb{C},2}(\Omega_{0}),\lambda_{\mathbb{C},2}(\Omega_{1})\}.
\]
Then $w$ satisfies
\begin{equation}\label{eq:w-viscosity}
\sigma_{2}(w_{i\bar{j}}-w_{i}w_{\bar{j}})\leq\Lambda
\end{equation}
in the admissible-test-function viscosity sense.
\end{lemma}

\begin{proof}
Let $\varphi\in C^{2}$ touch $w$ from below at $z_0$ and assume that it is an admissible test function, i.e.
\[
\varphi_{i\bar{j}}(z_0)-\varphi_{i}(z_0)\varphi_{\bar{j}}(z_0)\in\Gamma_{2}.
\]
Since $\varphi$ touches $w$ from below,
\[
D\varphi(z_0)=Dw(z_0),\qquad
D^{2}_{\mathbb{R}}\varphi(z_0)\leq D^{2}_{\mathbb{R}}w(z_0).
\]
Let $x\in\Omega_{0}$ and $y\in\Omega_{1}$ be the points realizing \eqref{eq:inf-conv}. By \eqref{eq:grad-agree}, the complex gradients agree:
\[
\partial v_{0}(x)=\partial v_{1}(y)=\partial w(z_0)=\partial\varphi(z_0)=:q.
\]
Put
\[
H_{i}=D^{2}_{\mathbb{R}}v_{i},\qquad H_{w}=D^{2}_{\mathbb{R}}w.
\]
By the eigenvalue equations,
\[
\sigma_{2}(C(H_{i})-qq^{*})
=\lambda_{\mathbb{C},2}(\Omega_{i})\leq\Lambda,\qquad i=0,1.
\]
Thus
\[
H_{i}^{-1}\in K_{q,\Lambda},\qquad i=0,1.
\]
By \eqref{eq:hessian-harmonic} and Theorem \ref{thm:rank-one},
\[
H_{w}^{-1}=(1-t)H_{0}^{-1}+tH_{1}^{-1}\in K_{q,\Lambda}.
\]
Therefore
\[
\sigma_{2}(C(H_{w})-qq^{*})\leq\Lambda.
\]
Finally,
\[
C(H_{w})-qq^{*}
=
(C(D^{2}_{\mathbb{R}}\varphi)-qq^{*})
+C(H_{w}-D^{2}_{\mathbb{R}}\varphi).
\]
The first term lies in $\Gamma_{2}$ by admissibility of the test function, and the second term is positive semidefinite. Since $\sigma_{2}$ is increasing on $\Gamma_{2}$ in positive semidefinite directions,
\[
\sigma_{2}(\varphi_{i\bar{j}}-\varphi_{i}\varphi_{\bar{j}})(z_0)
\leq
\sigma_{2}(w_{i\bar{j}}-w_{i}w_{\bar{j}})(z_0)
\leq\Lambda.
\]
This is the desired viscosity inequality.
\end{proof}

\subsection{Admissible replacement}

Let
\[
\tilde{u}=-e^{-w}.
\]
Lemma \ref{lem:viscosity} is equivalent to saying that
\begin{equation}\label{eq:tilde-viscosity}
\sigma_{2}(\tilde{u}_{i\bar{j}})\leq\Lambda(-\tilde{u})^{2}
\end{equation}
in the admissible-test-function viscosity sense. Salani uses this precise mechanism in the real Hessian setting: one tests only by admissible test functions and therefore does not need to prove that the convex envelope or infimal convolution is itself classically admissible \cite{salani2012convexity}.

We now convert this viscosity supersolution into a Rayleigh quotient estimate using an admissible replacement. For $\varepsilon>0$, solve the auxiliary Dirichlet problem
\begin{equation}\label{eq:auxiliary}
\begin{cases}
\sigma_{2}((\phi_{\varepsilon})_{i\bar{j}})
=\Lambda(-\tilde{u})^{2}+\varepsilon, & z\in\Omega_{t},\\
\phi_{\varepsilon}=0, & z\in\partial\Omega_{t},\\
\phi_{\varepsilon}<0,\qquad ((\phi_{\varepsilon})_{i\bar{j}})\in\Gamma_{2}.&
\end{cases}
\end{equation}
Since $\Omega_t$ is real strictly convex, it is strongly pseudoconvex and, in particular, strongly $2$-pseudoconvex. Hence there exists a smooth defining function $\rho$ such that
\[
\begin{gathered}
\Omega_{t}=\{\rho<0\},\\
\rho=0\quad\text{on }\partial\Omega_{t},\qquad
\left( \rho_{i\bar{j}} \right)\in\Gamma_{2}.
\end{gathered}
\]
Therefore, for a sufficiently large constant $A$, the function $A\rho$ is a subsolution of equation \eqref{eq:auxiliary}. Hence, by Theorem 1.1 in \cite{cp22}, equation \eqref{eq:auxiliary} has a $2$-admissible solution.

\begin{lemma}\label{lem:replacement}
For every $\varepsilon>0$,
\[
\phi_{\varepsilon}\leq\tilde{u}\quad\text{in }\Omega_{t}.
\]
\end{lemma}

\begin{proof}
Assume by contradiction that $\phi_{\varepsilon}-\tilde{u}$ has a positive interior maximum $m$ at $z_{0}$. Then
\[
\psi:=\phi_{\varepsilon}-m
\]
touches $\tilde{u}$ from below at $z_{0}$. Since $\phi_{\varepsilon}$ is admissible, so is $\psi$. By the viscosity inequality \eqref{eq:tilde-viscosity},
\[
\sigma_{2}(\psi_{i\bar{j}})(z_{0})\leq\Lambda(-\psi(z_{0}))^{2}.
\]
But $\psi(z_{0})=\tilde{u}(z_{0})$ and $\psi_{i\bar{j}}=(\phi_{\varepsilon})_{i\bar{j}}$. Hence, using \eqref{eq:auxiliary},
\[
\Lambda(-\tilde{u}(z_{0}))^{2}+\varepsilon
=
\sigma_{2}((\phi_{\varepsilon})_{i\bar{j}})(z_{0})
\leq
\Lambda(-\tilde{u}(z_{0}))^{2},
\]
which is impossible. Thus $\phi_{\varepsilon}\leq\tilde{u}$.
\end{proof}

\subsection{Proof of the Brunn--Minkowski inequality}

\begin{proof}[Proof of Theorem \ref{thm:BM}]
From Lemma \ref{lem:replacement},
\[
-\phi_{\varepsilon}\geq-\tilde{u}>0.
\]
Since $\phi_{\varepsilon}$ is admissible, the Rayleigh quotient gives
\[
\lambda_{\mathbb{C},2}(\Omega_{t})
\leq
\frac{\displaystyle\int_{\Omega_{t}}(-\phi_{\varepsilon})
\sigma_{2}((\phi_{\varepsilon})_{i\bar{j}})\,dV}
{\displaystyle\int_{\Omega_{t}}(-\phi_{\varepsilon})^{3}\,dV}.
\]
Using \eqref{eq:auxiliary},
\[
\lambda_{\mathbb{C},2}(\Omega_{t})
\leq
\Lambda
\frac{\displaystyle\int_{\Omega_{t}}(-\phi_{\varepsilon})(-\tilde{u})^{2}\,dV}
{\displaystyle\int_{\Omega_{t}}(-\phi_{\varepsilon})^{3}\,dV}+\varepsilon \frac{\displaystyle\int_{\Omega_{t}}(-\phi_{\varepsilon})dV}
{\displaystyle\int_{\Omega_{t}}(-\phi_{\varepsilon})^{3}\,dV}.
\]
Since $-\tilde{u}\leq-\phi_{\varepsilon}$, the quotient on the right is at most $1$. Hence
\[
\lambda_{\mathbb{C},2}(\Omega_{t})\leq\Lambda+\varepsilon \frac{\displaystyle\int_{\Omega_{t}}(-\phi_{\varepsilon})dV}
{\displaystyle\int_{\Omega_{t}}(-\phi_{\varepsilon})^{3}\,dV}.
\]
By H\"older's inequality,
\[
\int_{\Omega_{t}} -\phi_{\varepsilon} dV\leqslant \left( \int_{\Omega_{t}} \left( -\phi_{\varepsilon} \right)^{3} dV \right)^{\frac{1}{3}} |\Omega_{t} |^{\frac{2}{3}}.
\]
Hence
\[
\frac{\int_{\Omega_{t}} -\phi_{\varepsilon} dV}{\int_{\Omega_{t}} \left( -\phi_{\varepsilon} \right)^{3} dV} \leqslant \frac{1}{\left( \int_{\Omega_{t}} \left( -\phi_{\varepsilon} \right)^{3} dV \right)^{\frac{2}{3}}} |\Omega_{t} |^{\frac{2}{3}}\leqslant \frac{1}{\left( \int_{\Omega_{t}} \left( -\tilde{u} \right)^{3} dV \right)^{\frac{2}{3}}} |\Omega_{t} |^{\frac{2}{3}}.
\]
Therefore
\[
\lambda_{\mathbb{C},2}(\Omega_{t})\leq\Lambda+\varepsilon \frac{1}{\left( \int_{\Omega_{t}} \left( -\tilde{u} \right)^{3} dV \right)^{\frac{2}{3}}} |\Omega_{t} |^{\frac{2}{3}}.
\]
Letting $\varepsilon\downarrow0$ gives the max inequality
\begin{equation}\label{eq:max-ineq}
\lambda_{\mathbb{C},2}((1-t)\Omega_{0}+t\Omega_{1})
\leq
\max\{\lambda_{\mathbb{C},2}(\Omega_{0}),\lambda_{\mathbb{C},2}(\Omega_{1})\}.
\end{equation}

Finally, use homogeneity. Since
\[
\lambda_{\mathbb{C},2}(r\Omega)=r^{-4}\lambda_{\mathbb{C},2}(\Omega),
\]
set
\[
L=(1-t)\lambda_{\mathbb{C},2}(\Omega_{0})^{-1/4}
+t\lambda_{\mathbb{C},2}(\Omega_{1})^{-1/4}
\]
and
\[
\Omega'_{i}=\lambda_{\mathbb{C},2}(\Omega_{i})^{1/4}\Omega_{i},\qquad i=0,1.
\]
Then $\lambda_{\mathbb{C},2}(\Omega'_{i})=1$. Choose
\[
t'=\frac{t\lambda_{\mathbb{C},2}(\Omega_{1})^{-1/4}}{L}.
\]
A direct computation gives
\[
(1-t')\Omega'_{0}+t'\Omega'_{1}
=
\frac{1}{L}\bigl((1-t)\Omega_{0}+t\Omega_{1}\bigr).
\]
Applying \eqref{eq:max-ineq} to $\Omega'_{0}$ and $\Omega'_{1}$ yields
\[
\lambda_{\mathbb{C},2}\left(\frac{1}{L}\Omega_{t}\right)\leq1.
\]
Using homogeneity once more,
\[
L^{4}\lambda_{\mathbb{C},2}(\Omega_{t})\leq1.
\]
Equivalently,
\[
\lambda_{\mathbb{C},2}(\Omega_{t})^{-1/4}\geq L,
\]
which is exactly \eqref{eq:BM}.
\end{proof}

\section{Appendix}
Let $\Omega\subset\mathbb{C}^{n}$ be a bounded $C^{\infty}$-smooth, real uniformly strictly convex domain. Consider the principal complex $\sigma_{2}$-Hessian eigenvalue problem
\begin{equation}\label{eq:eigen-problem}
\begin{cases}
\sigma_{2}(u_{i\bar{j}})=\lambda_{\mathbb{C},2}(\Omega)(-u)^{2}, & z\in\Omega,\\
u<0, & z\in\Omega,\\
u=0, & z\in\partial\Omega,
\end{cases}
\end{equation}
where the solution is taken in the complex $\Gamma_{2}$ branch:
\[
(u_{i\bar{j}})(z)\in\Gamma_{2},\qquad z\in\Omega.
\]
Here
\[
\Gamma_{2}=\{A\in\mathbb{C}^{n\times n},A=A^{*}:\sigma_{1}(A)>0,\ \sigma_{2}(A)>0\}.
\]
For the complex $\sigma_{2}$ setting, Theorem \ref{ass:spectral} can be stated as follows.

\begin{theorem}\label{ass:spectral2}
For every bounded $C^{\infty}$-smooth, real uniformly strictly convex domain $\Omega\subset\mathbb{C}^{n}$, problem \eqref{eq:eigen-problem} has a negative $2$-admissible principal eigenfunction $u$, unique up to positive multiplicative constants, and
\[
\lambda_{\mathbb{C},2}(\Omega)>0,\qquad
u\in C^{1,1}\left( \bar{\Omega} \right)\cap C^{\infty}(\Omega).
\]
Moreover, if $\{\Omega_{t}\}_{t\in[0,1]}$ is a family of $C^{\infty}$-smooth, real uniformly strictly convex domains whose boundary geometry is uniformly controlled and depends continuously on $t$, then under the normalization
\begin{equation}\label{eq:normalization}
\int_{\Omega_{t}}(-u_{t})^{3}\,dV=1,
\end{equation}
one has
\[
t_{j}\to t_{0}
\quad\Longrightarrow\quad
\lambda_{\mathbb{C},2}(\Omega_{t_{j}})
\to
\lambda_{\mathbb{C},2}(\Omega_{t_{0}})
\]
and
\[
u_{t_{j}}\to u_{t_{0}}
\quad\text{in }C^{2}_{\mathrm{loc}}(\Omega_{t_{0}}).
\]
\end{theorem}

The remainder of this appendix explains why this assumption is supported by existing results and standard arguments.

\subsection{Real Convexity and Pseudoconvexity}

We first explain why the geometric condition in Theorem \ref{ass:spectral2} lies within the scope of the complex Hessian spectral theory.

\begin{lemma}\label{lem:real-to-pseudo}
Let $\Omega\subset\mathbb{C}^{n}\simeq\mathbb{R}^{2n}$ be a bounded $C^{2}$-smooth, real uniformly strictly convex domain. Then $\Omega$ is strictly pseudoconvex. More strongly, it is strongly $m$-pseudoconvex for every $1\leq m\leq n$, and in particular strongly $2$-pseudoconvex.
\end{lemma}

\begin{proof}
Choose a defining function $\rho$ such that
\[
\Omega=\{\rho<0\},\qquad
\partial\Omega=\{\rho=0\},\qquad
D\rho\neq0\quad\text{on }\partial\Omega.
\]
Real uniform strict convexity means that there exists $c_{0}>0$ such that for every $x\in\partial\Omega$ and every real tangent vector $\tau\in T_{x}\partial\Omega$,
\[
D^{2}\rho(x)[\tau,\tau]\geq c_{0}|\tau|^{2},
\]
with the exterior normal sign convention.

The Levi form is the Hermitian form obtained by restricting the real Hessian to complex tangential directions. More explicitly, if $\xi\in T^{1,0}_{x}\partial\Omega$ is a complex tangent vector, then its associated real vector belongs to a complex-invariant subspace of $T_{x}\partial\Omega$. Since the real second fundamental form is positive definite on all real tangential directions, it is positive definite in particular on complex tangential directions. Therefore the Levi form is positive definite, and $\Omega$ is strictly pseudoconvex.

Strong $m$-pseudoconvexity requires the first $m$ elementary symmetric functions of the Levi eigenvalues to lie in the positive cone. Since the Levi form is positive definite, its eigenvalues belong to $\Gamma_{m}$ for every $1\leq m\leq n$. Hence $\Omega$ is strongly $m$-pseudoconvex for every $m$, and in particular for $m=2$.
\end{proof}

\subsection{Existence, Uniqueness, and Regularity}
Define the admissible test-function class
\[
\mathcal{A}_{2}(\Omega)
=
\left\{
\varphi\in C^{2}(\Omega)\cap C(\overline{\Omega}):
\varphi<0\ \text{in }\Omega,\ 
\varphi=0\ \text{on }\partial\Omega,\ 
(\varphi_{i\bar{j}})\in\Gamma_{2}
\right\}.
\]
The formal Rayleigh quotient is
\begin{equation}\label{eq:rayleigh-app}
R_{\Omega}(\varphi)
=
\frac{\displaystyle\int_{\Omega}(-\varphi)\sigma_{2}(\varphi_{i\bar{j}})\,dV}
{\displaystyle\int_{\Omega}(-\varphi)^{3}\,dV}.
\end{equation}
Thus
\begin{equation}\label{eq:eigenvalue}
\lambda_{\mathbb{C},2}(\Omega)
=
\inf_{\varphi\in\mathcal{A}_{2}(\Omega)}R_{\Omega}(\varphi).
\end{equation}
This definition agrees with the variational eigenvalue in the complex Hessian energy space. Badiane--Zeriahi \cite{badiane2023variational} used variational methods to establish the first eigenvalue and eigenfunction for complex Hessian operators on bounded $m$-hyperconvex domains. Chu--Liu--McCleerey \cite{chu2024eigenvalue} further established $C^{1,1}$ regularity, uniqueness, and the variational formula for the first eigenvalue on strongly $m$-pseudoconvex manifolds. In the present setting, we take $m=2$.\\
Theorem 1.1 in \cite{chu2024eigenvalue}, together with Lemma \ref{lem:real-to-pseudo}, yields the following conclusion.
\begin{theorem}[Existence, Uniqueness, and Regularity]\label{ass:spectral3}
For every bounded $C^{\infty}$-smooth, real uniformly strictly convex domain $\Omega\subset\mathbb{C}^{n}$, problem \eqref{eq:eigen-problem} has a negative $2$-admissible principal eigenfunction $u$, unique up to positive multiplicative constants, and
\[
\lambda_{\mathbb{C},2}(\Omega)>0,\qquad
u\in C^{1,1}\left( \bar{\Omega} \right)\cap C^{\infty}(\Omega).
\]
\end{theorem}

\subsection[Local C2 Continuous Dependence Along Domain Deformations]{$C^{2}_{\mathrm{loc}}$ Continuous Dependence Along Domain Deformations}

We now prove the part of Theorem \ref{ass:spectral2} that is easiest to overlook: the normalized principal eigenfunction depends continuously on the domain in $C^{2}_{\mathrm{loc}}$ along smooth deformations.

\begin{proposition}[Deformation continuity]\label{prop:deformation}
Let $\{\Omega_{t}\}_{t\in[0,1]}$ be a family of bounded $C^{\infty}$-smooth, real uniformly strictly convex domains whose boundary geometry is uniformly controlled and depends continuously on $t$. Let $u_{t}$ be the principal eigenfunction normalized by
\[
\int_{\Omega_{t}}(-u_{t})^{3}\,dV=1,
\]
and set
\[
\lambda_{t}=\lambda_{\mathbb{C},2}(\Omega_{t}).
\]
Then for every $t_{0}\in[0,1]$ and every $K\Subset\Omega_{t_{0}}$,
\[
\lambda_{t}\to\lambda_{t_{0}},\qquad
u_{t}\to u_{t_{0}}\quad\text{in }C^{2}(K)
\]
as $t\to t_{0}$.
\end{proposition}

\begin{proof}
We divide the proof into five steps.

First, compact subsets remain inside nearby domains. Since $\partial\Omega_{t}$ varies continuously with $t$, if $K\Subset\Omega_{t_{0}}$, then there exists $\delta>0$ such that
\[
K\Subset\Omega_{t}
\]
whenever $|t-t_{0}|<\delta$.

Second, the eigenvalues have uniform upper and lower bounds. The uniform real strict convexity and boundedness of the deformation family give fixed balls $B_{r}$ and $B_{R}$ such that, after a uniform translation if necessary,
\[
B_{r}\subset\Omega_{t}\subset B_{R},\qquad 0\leq t\leq1.
\]
By domain monotonicity and the scaling law
\[
\lambda_{\mathbb{C},2}(s\Omega)=s^{-4}\lambda_{\mathbb{C},2}(\Omega),
\]
we obtain
\[
0<c\leq\lambda_{t}\leq C<\infty.
\]
Domain monotonicity and the scaling law are standard consequences of the variational definition; they can also be checked directly from \eqref{eq:rayleigh-app}.

Third, we have uniform interior higher-order estimates. Fix $K\Subset\Omega_{t_{0}}$. When $t$ is close to $t_{0}$, $K$ has uniformly positive distance from $\partial\Omega_{t}$. From the $C^{1,1}$ regularity theorem of Chu--Liu--McCleerey \cite{chu2024eigenvalue}, together with the uniform geometric control of the deformation family, the functions $u_{t}$ have locally uniform second derivative bounds. By the normalization and a compactness contradiction using the strong maximum principle, we claim that on a slightly smaller compact set one has
\[
-u_{t}\geq c_{K}>0.
\]
Thus the right-hand side $\lambda_{t}(-u_{t})^{2}$ has a positive lower bound on $K$, and the equation is uniformly elliptic there. By the interior estimates of \cite{chu2024eigenvalue}, for every $m\geq2$,
\[
\|u_{t}\|_{C^{m}(K)}\leq C_{K,m}
\]
uniformly for $t$ close to $t_{0}$. For the closedness argument, it is enough to take $m=4$ or $C^{2,\alpha}$ bounds.

Fourth, compactness gives a limiting equation. Let $t_{j}\to t_{0}$. By the preceding step, after passing to a subsequence,
\begin{equation}\label{eq1}
u_{t_{j}}\to u_{\infty}\quad\text{in }C^{2}(K)
\end{equation}
for every $K\Subset\Omega_{t_{0}}$. Since the eigenvalues are uniformly bounded, after passing to a further subsequence,
\[
\lambda_{t_{j}}\to\lambda_{\infty}.
\]
Passing to the limit in the equation gives
\begin{equation}\label{eq2}
\sigma_{2}((u_{\infty})_{i\bar{j}})
=
\lambda_{\infty}(-u_{\infty})^{2}
\quad\text{in }\Omega_{t_{0}}.
\end{equation}
The uniform bounds on the boundary geometry allow us to choose a constant $C$ such that
\[
\left\|u_{t}\right\|_{C^{1,1}\left( \overline{\Omega_{t}} \right)}\leqslant C,\qquad \forall t\in \left[ 0,1 \right].
\]
Choose a ball $B$ containing all $\Omega_t$ and extend each $u_t$ to $B$ with a uniform Lipschitz bound. Hence, by the Arzel\`a--Ascoli theorem, after passing to a subsequence we may assume that
\[
\lim_{j\longrightarrow \infty} \left\|u_{t_{j}}-u_{\infty}\right\|_{C^{0}\left( B \right)}=0.
\]
Together with the convergence of the domains, this implies
\[
u_{\infty}=0\quad\text{on }\partial\Omega_{t_{0}},\qquad
u_{\infty}<0\quad\text{in }\Omega_{t_{0}},
\]
and the normalization is preserved:
\[
\int_{\Omega_{t_{0}}}(-u_{\infty})^{3}\,dV=1.
\]
One may also invoke the stability theory for complex Hessian equations \cite{dinew2011apriori}.

Fifth, \eqref{eq1} and \eqref{eq2} show that $u_\infty$ is $2$-admissible. By the uniqueness of the principal eigenfunction in \cite{chu2024eigenvalue}, the normalized principal eigenfunction is unique. Therefore
\[
u_{\infty}=u_{t_{0}},\qquad \lambda_{\infty}=\lambda_{t_{0}}.
\]
Since every subsequential limit is the same, the whole family converges:
\[
u_{t}\to u_{t_{0}}\quad\text{in }C^{2}_{\mathrm{loc}}(\Omega_{t_{0}}),
\qquad
\lambda_{t}\to\lambda_{t_{0}}.
\]
\end{proof}

\small
\bibliographystyle{alpha}
\bibliography{reference}

@article{alvarez1997convex,
  author = {Alvarez, O. and Lasry, J.-M. and Lions, P.-L.},
  title = {Convex viscosity solutions and state constraints},
  journal = {J. Math. Pures Appl.},
  volume = {76},
  year = {1997},
  pages = {265--288}
}

@article{bian2009microscopic,
  author = {Bian, B. and Guan, P.},
  title = {A microscopic convexity principle for nonlinear partial differential equations},
  journal = {Invent. Math.},
  volume = {177},
  year = {2009},
  pages = {307--335}
}

@article{bian2010structural,
  author = {Bian, B. and Guan, P.},
  title = {A structural condition for microscopic convexity principle},
  journal = {Discrete Contin. Dyn. Syst.},
  volume = {28},
  year = {2010},
  pages = {789--807}
}

@article{bauschke2001hyperbolic,
  author = {Bauschke, H. H. and Guler, O. and Lewis, A. S. and Sendov, H. S.},
  title = {Hyperbolic polynomials and convex analysis},
  journal = {Canad. J. Math.},
  volume = {53},
  year = {2001},
  pages = {470--488}
}

@article{brascamp1976extensions,
  author = {Brascamp, H. J. and Lieb, E. H.},
  title = {On extensions of the {Brunn--Minkowski} and {Prekopa--Leindler} theorems, including inequalities for log-concave functions, and with an application to the diffusion equation},
  journal = {J. Funct. Anal.},
  volume = {22},
  year = {1976},
  pages = {366--389}
}

@article{badiane2023variational,
  author = {Badiane, P. and Zeriahi, A.},
  title = {A variational approach to the eigenvalue problem for complex {Hessian} operators},
  journal = {arXiv preprint arXiv:2306.04437},
  year = {2023}
}

@article{chu2024eigenvalue,
  author = {Chu, J. and Liu, Y. and McCleerey, N.},
  title = {The eigenvalue problem for the complex {Hessian} operator on {$m$}-pseudoconvex manifolds},
  journal = {J. Funct. Anal.},
  volume={290},
  year = {2026}
}

@article{colesanti2005brunn,
  author = {Colesanti, A.},
  title = {{Brunn--Minkowski} inequalities for variational functionals and related problems},
  journal = {Adv. Math.},
  volume = {194},
  year = {2005},
  pages = {105--140}
}

@article{colesanti1999hessian,
  author = {Colesanti, A. and Salani, P.},
  title = {Hessian equations in non-smooth domains},
  journal = {Nonlinear Anal.},
  volume = {38},
  year = {1999},
  pages = {803--812}
}

@article{dinew2011apriori,
  author = {Dinew, S. and Ko{\l}odziej, S.},
  title = {A priori estimates for the complex {Hessian} equations},
  journal = {Anal. PDE},
  volume={7},
  year = {2014},
  pages = {227-244}
}

@article{dinew2012liouville,
  author = {Dinew, S. and Ko{\l}odziej, S.},
  title = {{Liouville} and {Calabi--Yau} type theorems for complex {Hessian} equations},
  journal = {Amer. J. Math.},
  vloume={139},
  year = {2017},
  pages={403-415}
}

@article{garding1959inequality,
  author = {Garding, L.},
  title = {An inequality for hyperbolic polynomials},
  journal = {J. Math. Mech.},
  volume = {8},
  year = {1959},
  pages = {957--965}
}

@article{hou2010second,
  author = {Hou, Z. and Ma, X.-N. and Wu, D.},
  title = {A second order estimate for complex {Hessian} equations on a compact {K{\"a}hler} manifold},
  journal = {Math. Res. Lett.},
  volume = {17},
  year = {2010},
  pages = {547--561}
}

@article{liu2010brunn,
  author = {Liu, P. and Ma, X.-N. and Xu, L.},
  title = {A {Brunn--Minkowski} inequality for the {Hessian} eigenvalue in three-dimensional convex domain},
  journal = {Adv. Math.},
  volume = {225},
  year = {2010},
  pages = {1616--1633}
}

@article{renegar2006hyperbolic,
  author = {Renegar, J.},
  title = {Hyperbolic programs, and their derivative relaxations},
  journal = {Found. Comput. Math.},
  volume = {6},
  year = {2006},
  pages = {59--79}
}

@book{rockafellar1970convex,
  author = {Rockafellar, R. T.},
  title = {Convex Analysis},
  publisher = {Princeton University Press},
  address = {Princeton},
  year = {1970}
}

@article{salani2005monge,
  author = {Salani, P.},
  title = {A {Brunn--Minkowski} inequality for the {Monge--Amp$\grave{e}$re} eigenvalue},
  journal = {Adv. Math.},
  volume = {194},
  year = {2005},
  pages = {67--86}
}

@article{salani2012convexity,
  author = {Salani, P.},
  title = {Convexity of solutions and {Brunn--Minkowski} inequalities for {Hessian} equations in {$\mathbb{R}^{3}$}},
  journal = {Adv. Math.},
  volume = {229},
  year = {2012},
  pages = {1924--1948}
}

@article{urbas1990nonclassical,
  author = {Urbas, J. I. E.},
  title = {On the existence of nonclassical solutions for two classes of fully nonlinear elliptic equations},
  journal = {Indiana Univ. Math. J.},
  volume = {39},
  year = {1990},
  pages = {355--382}
}

@article{wang1994class,
  author = {Wang, X.-J.},
  title = {A class of fully nonlinear elliptic equations and related functionals},
  journal = {Indiana Univ. Math. J.},
  volume = {43},
  year = {1994},
  pages = {25--54}
}

@article{cp22,
  author = {Collins, Tristan C. and Picard, Sebastien},
  title = {The {Dirichlet} problem for the {$k$}-{Hessian} equation on a complex manifold},
  journal = {Amer. J. Math.},
  volume = {6},
  year = {2022},
  pages = {1641--1680}
}

@article{korevaar1983convex,
  author = {Korevaar, N.},
  title = {Convex solutions to nonlinear elliptic and parabolic boundary value problems},
  journal = {Indiana Univ. Math. J.},
  volume = {32},
  year = {1983},
  pages = {603--614}
}

@article{caffarelli1985convexity,
  author = {Caffarelli, L. and Friedman, A.},
  title = {Convexity of solutions of some semilinear elliptic equations},
  journal = {Duke Math. J.},
  volume = {52},
  year = {1985},
  pages = {431--455}
}

@article{limasa,
  author = {Jiahuan, Li and Xi-Nan, Ma and Paolo, Salani},
  title = {A Brunn--Minkowski inequality for the Hessian eigenvalue in convex domain},
  journal = {arXiv preprint arXiv:2606.22847},
  year = {2026}
}

@article{zdk,
  author = {Zhang, Dekai},
  title = {Regularity of degenerate k-Hessian equations on closed Hermitian manifolds.},
  journal = {Adv.Nonlinear Stud.},
  volume = {22},
  year = {2022},
  pages = {534-547}
}
{\small
\indent 
(Chuanqiang Chen) School of Mathematics and Statistics, Ningbo University, No. 818 Fenghua Road, Jiangbei District, 315211 Ningbo, P. R. China;
Email address: chenchuanqiang@nbu.edu.cn\\ \\
(Jiahuan Li) Department of Mathematics, University of Science and Technology of China, Hefei, 230026, Anhui Province, China.\;
Email address: jiahuan@mail.ustc.edu.cn\\ \\
(Xi-Nan Ma) Department of Mathematics, University of Science and Technology of China, Hefei, 230026, Anhui Province, China.\;
Email address: xinan@ustc.edu.cn\\ \\

}

\end{document}